\documentclass[11pt]{article}
\usepackage{amsfonts,amssymb,amsthm,amsmath,cite,bbm}
\usepackage{enumerate}
\usepackage{hyperref}

\hypersetup{
	colorlinks,%
	citecolor=black,%
	filecolor=black,%
	linkcolor=black,%
	urlcolor=black
}

\newcommand{\N}{\mathbb N} 
\newcommand{\R}{\mathbb R} 
\newcommand{\Rn}{\R^n} 

\newcommand{\CV}{\operatorname{Conv}(\Rn)}
\newcommand{\CVc}{\operatorname{Conv_{c}}(\Rn)} 
\newcommand{\CVcf}{\operatorname{Conv_{c}}(\Rn,\R)} 
\newcommand{\CVs}{\operatorname{Conv_{sc}}(\Rn)} 
\newcommand{\CVsf}{\operatorname{Conv_{sc}}(\Rn,\R)} 
\newcommand{\CVoin}{\operatorname{Conv_{(o)}}(\Rn)} 

\newcommand{\K}{{\mathcal K}}
\newcommand{\Kn}{\K^n} 
\newcommand{\Koin}{\Kn_{(o)}} 

\newcommand{\Q}{{\mathcal Q}}

\newcommand{\mx}{\mathbin{\vee}} 
\newcommand{\mn}{\mathbin{\wedge}} 

\newcommand{\Ind}{\mathrm{I}}

\renewcommand{\l}{\ell}

\renewcommand{\d}{\,\mathrm{d}}

\newcommand{\sumf}{\mathrm{sf}}

\newcommand\SLn{\operatorname{SL}(n)}

\newcommand{\dom}{\operatorname{dom}}
\newcommand{\interior}{\operatorname{int}}

\newcommand{\argmax}{\operatorname{argmax}}

\newcommand{\elim}{\operatorname{epi-lim}}

\newcommand{\oZ}{\operatorname{Z}}

\newcommand{\eto}{\stackrel{epi}{\longrightarrow}}

\newtheorem{lemma}{Lemma}[section]
\newtheorem{theorem}[lemma]{Theorem}
\newtheorem*{theorem*}{Theorem}

\newtheorem*{corollary*}{Corollary}

\theoremstyle{definition}

\theoremstyle{remark}
\newtheorem{remark}[lemma]{Remark}
\newtheorem*{remark*}{Remark}

\title{$\SLn$ Invariant Valuations on Super-Coercive Convex Functions}
\author{Fabian Mussnig}
\date{}

\begin{document}

\maketitle

\begin{abstract}
All non-negative, continuous, $\SLn$ and translation invariant valuations on the space of super-coercive, convex functions on $\Rn$ are classified. Furthermore, using the invariance of the function space under the Legendre transform, a classification of non-negative, continuous, $\SLn$ and dually translation invariant valuations is obtained. In both cases, different functional analogs of the Euler characteristic, volume and polar volume are characterized.
\bigskip

{\noindent
{\bf 2010 AMS subject classification:} 26B25 (46A40, 52A20, 52A41, 52B45).\\
{\bf Keywords:} valuations, convex functions, super-coercive, Legendre transform, $\SLn$ invariance, polar volume.
}
\end{abstract}

\section{Introduction and Main Results}
\noindent
At the Paris ICM in 1900, David Hilbert asked the following question: 
Given two polytopes with equal volume, can one of them be cut into finitely many pieces that can be used to yield the other? It was already known that this is possible in the 2-dimensional case but the higher dimensional cases were still open. In the same year, Max Dehn was able to construct two polytopes that have the same volume but cannot be cut and reassembled to yield each other. Thus the answer to Hilbert's question is no for dimensions greater or equal than 3. In his proof, Dehn used the so-called Dehn invariant and made substantial use of its valuation property. To be more precise, let $\Kn$ denote the space of convex bodies, i.e.\! compact, convex sets, in $\Rn$. A map $\mu:\Q^n \subseteq \Kn\to\R$ is called a \emph{valuation} whenever
$$\mu(K)+\mu(L)=\mu(K\cup L) + \mu(K\cap L)$$
for every $K,L\in\Q^n$ such that also $K\cup L, K\cap L \in\Q^n$. Since Dehn's proof, valuations have been studied extensively in convex and discrete geometry and a first classification result was established by Blaschke in the 1930s \cite{blaschke}. He proved that linear combinations of the Euler characteristic and the $n$-dimensional volume are the only continuous, $\SLn$ and translation invariant valuations on $\Kn$. Here, continuity is understood with respect to the Hausdorff metric.

Important generalizations of Blaschke's result have been obtained since then \cite{ludwig_reitzner_annals,ludwig_reitzner,ludwig_val_poly_origin_int,Alesker99}. Recently, Haberl and Parapatits generalized Blaschke's result to valuations defined on $\Koin$, the set of convex bodies in $\Rn$ that contain the origin in their interiors. Note, that by restricting to a smaller space, it is possible that more valuations appear in a classification result. In this case, not only the $n$-dimensional volume, $V_n$, and the Euler characteristic, $V_0$, were characterized, but also the \emph{polar volume} $V_n^*(K):=V_n(K^*)$. Here $K^*=\{x\in\Rn\,:\, x\cdot y \leq 1, \forall y\in K\}$ is the \emph{polar body} of $K\in\Koin$.

\begin{theorem}[\!\!\cite{haberl_parapatits_centro}]
\label{thm:haberl_parapatits}
For $n\geq 2$, a map $\mu:\Koin\to\R$ is a continuous and $\SLn$ invariant valuation if and only if there exist constants $c_0,c_1,c_2\in\R$ such that
\begin{equation}
\label{eq:haberl_parapatits}
\mu(K)=c_0 V_0(K)+c_1 V_n(K)+c_2 V_n^*(K)
\end{equation}
for every $K\in\Koin$.
\end{theorem}
Here, a valuation $\mu:\Koin\to\R$ is said to be \emph{$\SLn$ invariant} if $\mu(\phi K)=\mu(K)$ for every $\phi\in\SLn$ and $K\in\Koin$.

In recent years, the notion of valuation was extended to functions spaces. Let $\mathcal{S}$ be a space of (extended) real-valued functions on $\Rn$. We say that a map $\oZ:\mathcal{S}\to\R$ is a \emph{valuation} whenever
$$\oZ(u)+\oZ(v)=\oZ(u\vee v) + \oZ(u \wedge v)$$
for every $u,v\in\mathcal{S}$ such that also $u\vee v, u \wedge v\in\mathcal{S}$. Here, $u\vee v$ and $u\wedge v$ denote the pointwise maximum and minimum of the functions $u,v\in\mathcal{S}$, respectively. In particular, valuations on Sobolev spaces \cite{ludwig_fisher,ludwig_sobolev,ma}, $L^p$ spaces \cite{ludwig_covariance,ober_minkowski,tsang_val_on_lp,tsang_mink_val_on_lp}, on definable functions \cite{baryshnikov_ghrist_wright} and on quasi-concave functions \cite{bobkov_colesanti_fragala,colesanti_fragala,colesanti_lombardi,colesanti_lombardi_parapatits,milman_rotem} were studied and characterized. See also \cite{Alesker_Convex,colesanti_ludwig_mussnig_mink,colesanti_ludwig_mussnig_hess,villanueva,tradacete_villanueva_adv,tradacete_villanueva_imrn,wang_semi_val,ludwig_function}.

For convex functions, an analog of Blaschke's characterization of continuous, $\SLn$ and translation invariant valuations was established in \cite{colesanti_ludwig_mussnig}. More recently, this result was improved and a functional analog of Theorem~\ref{thm:haberl_parapatits} was found. Thereby, functional versions of the Euler characteristic and volume together with a new analog of the polar volume were characterized. In order to state this result, let $\CVcf$ denote the space of all convex, coercive functions $u:\Rn\to\R$. Here, a function $u$ is said to be \emph{coercive} if
$$\lim_{|x|\to\infty} u(x)=+\infty.$$
We equip $\CVcf$ with the topology associated to pointwise convergence (see also Section \ref{se:convex_functions}). A map $\oZ:\mathcal{S}\to\R$ is called \emph{translation invariant} if $\oZ(u\circ \tau^{-1})=\oZ(u)$ for every $u\in\mathcal{S}$ and translation $\tau$ on $\Rn$. Furthermore, $\oZ$ is said to be $\SLn$ invariant if $\oZ(u\circ \phi^{-1})=\oZ(u)$ for every $u\in\mathcal{S}$ and $\phi\in\SLn$.

\begin{theorem}[\!\!\cite{mussnig_advances}]
\label{thm:class_advances}
For $n\geq 2$, a map $\oZ:\CVcf\to[0,\infty)$ is a continuous, $\SLn$ and translation invariant valuation if and only if there exist continuous functions $\zeta_0,\zeta_1,\zeta_2:\R\to[0,\infty)$ where $\zeta_1$ has finite moment of order $n-1$ and $\zeta_2(t)=0$ for all $t\geq T$ with some $T\in\R$ such that
\begin{equation}
\label{eq:class_advances}
\oZ(u)=\zeta_0\big(\min\nolimits_{x\in\Rn} u(x)\big) + \int_{\Rn} \zeta_1\big(u(x)\big) \d x + \int_{\dom u^*} \zeta_2\big(\nabla u^*(x)\cdot x - u^*(x)\big) \d x
\end{equation}
for every $u\in\CVcf$.
\end{theorem}

Here, a function $\zeta:\R\to[0,\infty)$ has \emph{finite moment of order $n-1$} if $\int_0^{\infty} t^{n-1} \zeta(t) \d t < +\infty$. Note that for functions $u\in\CVcf$ the minimum is attained and hence finite. For a convex function $u$ on $\Rn$
$$u^*(x)=\sup\nolimits_{y\in\Rn} \big(x\cdot y - u(y)\big),\qquad x\in\Rn$$
denotes the \emph{Legendre transform} or \emph{convex conjugate} of $u$, where $x\cdot y$ denotes the inner product of $x,y\in\Rn$. Moreover, $\dom u^*=\{x\in\Rn\,:\, u^*(x) <+\infty\}$ denotes the \emph{domain of $u^*$}, which is needed since $u^*$ might attain the value $+\infty$. Lastly, $\nabla u^*$ denotes the \emph{gradient} of $u^*$. Note, that it follows from Rademacher's theorem (see for example \cite[Theorem 3.1.6]{federer}) that the convex function $u^*$ is differentiable a.e.\! on the interior of its domain.

\begin{remark*}
Observe, that \eqref{eq:haberl_parapatits} can be retrieved from \eqref{eq:class_advances} if $u$ is chosen to be $\|\cdot\|_K$, the norm with unit ball $K\in\Koin$.
\end{remark*}

We will show that the statement of Theorem~\ref{thm:class_advances} is still true on the space
$$\CVsf:=\{u:\Rn\to\R\,:\, u \text{ is convex and super-coercive}\}$$
where we say that a function $u$, defined on $\Rn$, is \emph{super-coercive} if
$$\lim_{|x|\to\infty} \frac{u(x)}{|x|}=+\infty.$$
It is a priori not clear that no new valuations appear on $\CVsf$. Note, that the proof of Theorem~\ref{thm:class_advances} made extensive use of functions that are coercive but not super-coercive. Furthermore, to the best of the author's knowledge, there does not seem to be an easy way to generalize the proof to the setting of super-coercive functions.

Moreover, we want to point out that the space $\Koin$ is invariant under the polarity transform, that is $\{K^*\,:\, K\in\Koin\}=\Koin$. Results of Artstein-Avidan and Milman \cite{artstein_milman_annals} show that the Legendre transform is the only natural, functional analog of the polarity transform on most spaces of convex functions. In contrast to the space $\CVcf$, the space $\CVsf$ is invariant under the Legendre transform, that is,
\begin{eqnarray*}
\{u^*\,:\, u\in\CVcf\} &\neq& \CVcf\\
\{u^*\,:\, u\in\CVsf\} &=& \CVsf.
\end{eqnarray*}
In that sense, $\CVsf$ seems to be a better functional analog of the space $\Koin$. For further details see Section~\ref{se:convex_functions}.

\begin{theorem}
\label{thm:main_result}
For $n\geq 2$, a map $\oZ:\CVsf\to[0,\infty)$ is a continuous, $\SLn$ and translation invariant valuation if and only if there exist continuous functions $\zeta_0,\zeta_1,\zeta_2:\R\to[0,\infty)$ where $\zeta_1$ has finite moment of order $n-1$ and $\zeta_2(t)=0$ for all $t\geq T$ with some $T\in\R$ such that
\begin{equation}
\label{eq:transl_inv_case}
\oZ(u)=\zeta_0\big(\min\nolimits_{x\in\Rn} u(x)\big) + \int_{\Rn} \zeta_1\big(u(x)\big) \d x + \int_{\Rn} \zeta_2\big(\nabla u^*(x)\cdot x - u^*(x)\big) \d x
\end{equation}
for every $u\in\CVsf$.
\end{theorem}

By using the invariance of $\CVsf$ under the Legendre transform we also obtain the following equivalent result. Let $\mathcal{S}$ be a space of super-coercive, convex functions on $\Rn$. A map $\oZ:\mathcal{S}\to\R$ is said to be \emph{dually translation invariant} if $\oZ(u+l)=\oZ(u)$ for every $u\in\mathcal{S}$ and every linear functional $l$ on $\Rn$. Equivalently, $\oZ$ is dually translation invariant if and only if $u\mapsto \oZ(u^*)$ is translation invariant for every $u$ such that $u^*\in\mathcal{S}$.

\renewcommand\thelemma{1.3*}
\begin{theorem}
\label{cor:main_result}
For $n\geq 2$, a map $\oZ:\CVsf\to[0,\infty)$ is a continuous, $\SLn$ and dually translation invariant valuation if and only if there exist continuous functions $\zeta_0,\zeta_1,\zeta_2:\R\to[0,\infty)$ where $\zeta_1$ has finite moment of order $n-1$ and $\zeta_2(t)=0$ for all $t\geq T$ with some $T\in\R$ such that
\begin{equation}
\label{eq:dually_transl_inv_case}
\oZ(u) = \zeta_0\big(u(0)\big) + \int_{\Rn} \zeta_1\big( u^*(x) \big)\d x + \int_{\Rn} \zeta_2\big(\nabla u(x)\cdot x - u(x) \big)\d x
\end{equation}
for every $u\in\CVsf$.
\end{theorem}

\begin{remark*}
The \emph{volume product} $V_n(K) V_n(K^*)$ is of significant interest in convex geometric analysis. In particular,
\begin{equation}
\label{eq:blaschke_santalo_bourgain_milman}
c^n V_n(B^n)^2 \leq V_n(K)V_n(K^*)\leq V_n(B^n)^2
\end{equation}
for every origin symmetric $K\in\Koin$, i.e. $K=-K$, where $B^n$ denotes the Euclidean unit ball and $c>0$ is an absolute constant. The right side of \eqref{eq:blaschke_santalo_bourgain_milman} is sharp with the maximizers being ellipsoids \cite{petty} and is also known as the Blaschke-Santal\'o inequality \cite{blaschke_affine_geometrie_VII,santalo}. The left side is due to Bourgain and Milman \cite{bourgain_milman} but the optimal constant $c$ is still not known. The famous Mahler conjecture states that the volume product is minimized for affine transforms of cubes (among others) and a proof for the two-dimensional case is due to Mahler \cite{mahler}. More recently, the conjecture was confirmed for the three-dimensional case \cite{iriyeh_shibata}, but the general case remains open.

Functional versions of \eqref{eq:blaschke_santalo_bourgain_milman} for log-concave functions were obtained in \cite{artstein_klartag_milman, artstein_slomka, ball_thesis, klartag_milman}. In particular, it was shown that
$$\left( \frac{2 \pi}{c}\right)^n \leq \int_{\Rn} \exp\big(-u(x)\big) \d x \int_{\Rn} \exp\big(-u^*(x)\big) \d x \leq (2 \pi)^n$$
for suitable convex functions $u$ on $\Rn$, where $c>0$ is again an absolute constant. Considering Theorem~\ref{thm:main_result} and Theorem~\ref{cor:main_result}, the question arises if a similar inequality can be obtained using the quantities
$$u \mapsto \int_{\dom u^*} \zeta\big(\nabla u^*(x)\cdot x - u^*(x)\big) \d x\quad \text{and/or} \quad u \mapsto \int_{\dom u} \zeta\big(\nabla u(x)\cdot x - u(x)\big)\d x$$
for suitable functions $\zeta:\R\to\R$ and convex functions $u$ on $\Rn$.
\end{remark*}

\section{Convex Functions}
\label{se:convex_functions}
\renewcommand\thelemma{\arabic{section}.\arabic{lemma}}
We will work in $n$-dimensional Euclidean space, $\Rn$. Let $\CV$ denote the space of all convex, proper, lower semicontinuous functions $u:\Rn\to(-\infty,\infty]$, where we call a function $u$ on $\Rn$ \emph{proper} if $u\not\equiv +\infty$. We will consider the following subsets of $\CV$:
\begin{eqnarray*}
\CVc &=& \{u\in\CV \,:\,u \text{ is coercive}\}\\
\CVs &=& \{u\in\CV \,:\, u \text{ is super-coercive}\}\\
\CVoin &=& \{u\in\CV \,:\, 0 \in \interior \dom u\}
\end{eqnarray*}
where $\interior A$ denotes the \emph{interior} of the set $A\subseteq \Rn$. Furthermore, $\CVcf$ and $\CVsf$ will denote the sets of functions in $\CVc$ and $\CVs$, respectively, that only take values in $\R$, i.e.\! functions that do not attain the value $+\infty$.

For $u\in\CV$ and $t\in\R$ we will write
$$\{u\leq t\}=\{x\in\Rn\,:\, u(x)\leq t\}$$
for the \emph{sublevel sets} of $u$. Since $u$ is convex and lower semicontinuous, the sets $\{u\leq t\}$ are convex and closed. Moreover, since $u$ is proper, there exists $t\in\R$ such $\{u\leq t\}\neq \emptyset$. If in addition $u$ is coercive, then all sublevel sets of $u$ are bounded. In particular $\{u\leq t\}\in\Kn$ for all $t\geq \min_{x\in\Rn} u(x)$.

For $K\in\Kn$ we will denote by
$$\Ind_K(x) = \begin{cases}
0,\quad & x\in K\\
+\infty, \quad & x \notin K
\end{cases} $$
the \emph{(convex) indicator function of $K$}. Observe, that $\{\Ind_K\leq t\} = K$ for all $t\geq 0$ and $\Ind_K\in\CVs$ for all $K\in\Kn$.

The space $\CV$ and its subspaces will be equipped with the topology due to epi-convergence. Here, we say that a sequence $u_k\in\CV$, $k\in\N$ is \emph{epi-convergent} to $u\in\CV$ if the following two conditions hold for all $x\in\Rn$:
\begin{enumerate}
	\item[(i)] For every sequence $x_k$ that converges to $x$,
	$$u(x)\leq \liminf\nolimits_{k\to\infty} u_k(x_k).$$
	\item[(ii)] There exists a sequence $x_k$ that converges to $x$ such that
	$$u(x)=\lim\nolimits_{k\to\infty} u_k(x_k).$$
\end{enumerate}
If a sequence $u_k$ is epi-convergent to $u$, we will write $u=\elim_{k\to\infty} u_k$ and $u_k \eto u$.

For elements in $\CV$, epi-convergence coincides with local uniform convergence a.e. The only exceptions occur at the boundary of the domain of the limit function.

\begin{theorem}[\!\!\cite{rockafellar_wets}, Theorem 7.17]
\label{thm:epi_conv_is_nice}
For any epi-convergent sequence of convex functions $u_k:\Rn\to(-\infty,\infty]$ the limit function $u=\elim_{k\to\infty} u_k$ is convex. Moreover, under the assumption the $u:\Rn\to(-\infty,\infty]$ is convex and lower semicontinuous such that $\dom u$ has nonempty interior, the following are equivalent:
\begin{enumerate}
	\item[(a)] $u=\elim_{k\to\infty} u_k.$
	\item[(b)] $u_k(x)\to u(x)$ for all $x\in D$, where $D$ is a dense subset of $\Rn$.
	\item[(c)] $u_k$ converges uniformly to $u$ on every compact set $C\subset \Rn$ that does not contain a boundary point of $\dom u$.
\end{enumerate}
\end{theorem}

\begin{remark}
It is a consequence of Theorem~\ref{thm:epi_conv_is_nice} that epi-convergence coincides with pointwise convergence on $\CVcf$ and $\CVsf$. See also \cite[Example 5.13]{dal_maso}.
\end{remark}

For functions in $\CVc$, epi-convergence also corresponds to Hausdorff convergence of sublevel sets. In the following we say that $\{u_k\leq t \} \to \emptyset$ as $k\to\infty$ if there exists $k_0\in\N$ such that $\{u_k\leq t\}=\emptyset$ for all $k\geq k_0$.

\begin{lemma}[\!\!\cite{colesanti_ludwig_mussnig}, Lemma 5 and \cite{beer_rockafellar_wets}, Theorem 3.1]
\label{le:hd_conv_lvl_sets}
Let $u_k,u\in\CVc$. If $u_k\eto u$, then $\{u_k\leq t\} \to \{u\leq t\}$ as $k\to+\infty$ for every $t\in\R$ with $t\neq \min_{x\in\Rn} u(x)$. Furthermore, if for every $t\in\R$ there exists a sequence $t_k\to t$ such that $\{u_k\leq t_k\}\to\{u\leq t\}$, then $u_k\eto u$.
\end{lemma}

Next, we want to recall some results about the convex conjugate or Legendre transform
$$u^*(x)=\sup\nolimits_{y\in\Rn} \big( x\cdot y - u(y)\big)$$
for every $x\in\Rn$ and $u\in\CV$.

\begin{lemma}[\!\!\cite{schneider}, Theorem 1.6.13]
\label{le:u_starstar_is_u}
If $u\in\CV$, then also $u^*\in\CV$ and $u^{**}=u$.
\end{lemma}

The following is easy to see and follows directly from the definition of the convex conjugate. See also \cite[Section 3]{mussnig_advances}

\begin{lemma}
\label{le:conjugate_transl_sln_inv}
Let $\oZ:\mathcal{S}\to \R$, with $\mathcal{S}\subseteq \CV$. The operator $\oZ$ is translation invariant if and only if $u\mapsto \oZ(u^*)$ is dually translation invariant and $\oZ$ is $\SLn$ invariant if and only if $u\mapsto \oZ(u^*)$ is $\SLn$ invariant, where $u$ is such that $u^*\in\mathcal{S}$.
\end{lemma}

The next lemma shows that the Legendre transform is compatible with the valuation property.

\begin{lemma}[\!\!\cite{colesanti_ludwig_mussnig_hess}, Lemma 3.4, Proposition 3.5]
\label{le:conjugate_is_a_val}
Let $u,v\in\CV$. If $u\wedge v$ is convex, then so is $u^* \wedge v^*$. Furthermore,
$$(u\wedge v)^* = u^* \vee v^* \quad \text{and} \quad (u \vee v)^* = u^* \wedge v^*.$$
\end{lemma}

The following result establishes a connection between coercivity properties of a function and the domain of its conjugate.

\begin{lemma}[\!\!\cite{rockafellar_wets}, Theorem 11.8]
\label{le:conjugate_coercive}
For $u\in\CV$ the following hold true:
	\begin{itemize}
		\item $u$ is coercive if and only if $0\in\interior \dom u^*$.
		\item $u$ is super-coercive if and only if $\dom u^* = \Rn$.
	\end{itemize}
\end{lemma}

We will also need the following theorem due to Wijsman, which shows that the Legendre transform is a continuous operation (see, for example, \cite[Theorem 11.34]{rockafellar_wets}).

\begin{theorem}
\label{thm:wijsman}
If $u_k, u \in\CV$, then $u_k \eto u$ if and only if $u_k^* \eto u^*$.
\end{theorem}

Let $u\in\CV$ and $x\in\Rn$. We call a vector $y\in\Rn$ a \emph{subgradient of $u$ at $x$} if
$$u(z)\geq u(x) + (z-x)\cdot y$$
for every $z\in\Rn$. The set of all subgradients of $u$ at $x$ is called the \emph{subdifferential of $u$ at $x$} and denoted by $\partial u(x)$. Note, that $\partial u(x)$ might be empty. Furthermore, if $u$ is differentiable at $x$, then the only possible subgradient of $u$ at $x$ is the gradient itself and $\partial u(x)=\{\nabla u(x)\}$.

\begin{lemma}[\!\!\cite{rockafellar}, Theorem 23.5]
\label{le:conjugate_subgradients}
For $u\in\CV$ and $x,y\in\Rn$ the following are equivalent:
\begin{itemize}
	\item $y\in\partial u(x)$,
	\item $x\in\partial u^*(y)$,
	\item $x\cdot y = u(x)+u^*(y)$,
	\item $x\in\argmax_{z\in\Rn} \big( y \cdot z - u(z)\big)$
	\item $y\in\argmax_{z\in\Rn} \big( x\cdot z - u^*(z)\big)$.
\end{itemize}
\end{lemma}
Here, $\argmax_{z\in V} f(z)$ denotes the points in the set $V$ at which the function values of $f$ are maximized on $V$.

For further results on convex functions as well as convex geometry in general we refer to the books of Gruber \cite{gruber}, Rockafellar \& Wets \cite{rockafellar_wets} and Schneider \cite{schneider}.

\section{Valuations on Convex Functions}
In this section we discuss the operators that appear in Theorem~\ref{thm:main_result}. In the following we say that a valuation $\oZ:\mathcal{S}\to\R$, where $\mathcal{S}$ is a space of (extended) real-valued functions on $\Rn$, is \emph{homogeneous of degree $p\in\R$} if $\oZ(u_{\lambda})=\lambda^p \oZ(u)$ for every $u\in\mathcal{S}$ and $\lambda >0$, where $u_\lambda(x):=u(x/\lambda)$ for $x\in\Rn$.

The following operator is a functional analog of the Euler characteristic.
\begin{lemma}[\!\!\cite{colesanti_ludwig_mussnig}, Lemma 12]
\label{le:min_is_a_val}
For a continuous function $\zeta:\R\to\R$ the map
$$u\mapsto \zeta\big(\min\nolimits_{x\in\Rn} u(x)\big)$$
defines a continuous, $\SLn$ and translation invariant valuation on $\CVc$ that is homogeneous of degree $0$.
\end{lemma}

By combining the last operator with the Legendre transform we obtain a dually translation invariant valuation.

\begin{lemma}[\!\!\cite{mussnig_advances}, Lemma 4.9]
\label{le:min_is_a_val_polar}
For a continuous function $\zeta:\R\to\R$ the map
\begin{equation}
\label{eq:min_is_a_val_polar}
u\mapsto \zeta\big(\min\nolimits_{x\in\Rn} u^*(x)\big) = \zeta\big(-u(0)\big)
\end{equation}
defines a continuous, $\SLn$ and dually translation invariant valuation on $\CVoin$ that is homogeneous of degree 0.
\end{lemma}

\begin{remark}
Note, that $u\mapsto \zeta\big(-u(0)\big)$ is also well defined on
$$\{u\in\CV \,:\, 0 \in \dom u\} \supset \CVoin$$
and in \cite[Lemma 4.9]{mussnig_advances} it is wrongfully claimed that even on this larger space \eqref{eq:min_is_a_val_polar} still defines a continuous, $\SLn$ and dually translation invariant valuation. To see that this valuation is not continuous anymore let $\l_K\in\CVc$ be defined via
$$\{\l_K\leq t\} = t K$$
for every $K\in\Kn$ with $0\in K$ and $t \geq 0$. Let $P_k:=[-1/k,1]\times[-1,1]^{n-1}$ and let $u_k=\l_{P_k}\circ \tau_k^{-1}$, where $\tau_k(x)=x+e_1/k$ for $x\in\Rn$ and $k\in\N$, where $e_1$ denotes the first vector of the standard basis of $\Rn$. Observe, that $u_k(0)=1$ for every $k\in\N$. By Lemma~\ref{le:hd_conv_lvl_sets} it is easy to see that $u_k\eto \l_P$ as $k\to\infty$, where $P:=[0,1]\times[-1,1]^{n-1}$ but $\l_P(0)=0$. In particular, $0\in \dom \l_P$ but $\l_P  \notin \CVoin$.
\end{remark}

\begin{lemma}[\!\!\cite{colesanti_ludwig_mussnig}, Lemma 16]
\label{le:int_is_a_val}
For a continuous function $\zeta:\R\to[0,\infty)$ with finite moment of order $n-1$, the map
$$u\mapsto \int_{\dom u} \zeta\big(u(x)\big) \d x$$
defines a non-negative, continuous, $\SLn$ and translation invariant valuation on $\CVc$ that is homogeneous of degree $n$.
\end{lemma}

The next lemma shows that the moment condition for the function $\zeta$ is necessary, even if one restricts to super-coercive, convex functions. 

\begin{lemma}
\label{le:no_moment_no_finite}
If $\zeta\in C(\R)$ is non-negative such that $\int_0^{\infty} t^{n-1} \zeta(t) \d t = \infty$, then there exists $u_\zeta\in\CVsf$ such that
$$\int_{\Rn} \zeta\big(u_\zeta(x)\big) \d x = \infty.$$
\end{lemma}
\begin{proof}
Let $t_0=0$. By the assumption on $\zeta$ there exists numbers $t_k >0$, $k\in\N$ such that $\int_0^{t_k} t^{n-1} \zeta(t) \d t \geq k$ and $t_k-t_{k-1} \geq 1$. Let $r_0=0$ and let
$$r_k=\frac{t_k-t_{k-1}}{k^{1/n}}+r_{k-1}$$
for every $k\geq 1$. We now set $v_\zeta(r)=k^{1/n}(r-r_{k-1})+t_{k-1}$ for every $r_{k-1}\leq r <r_k$ and for every $k\geq 1$. Note, that by the choice of $t_k$ we have $\lim_{k\to\infty} r_k = +\infty$, which shows that $v_\zeta(r)$ is finite. Furthermore, it is easy to see that $v_\zeta(r_{k-1})=t_{k-1}$ and
$$\lim_{r\to r_k} v_\zeta(r) = k^{1/n}\left( \left(\frac{t_k-t_{k-1}}{k^{1/n}}+r_{k-1}\right) - r_{k-1}\right) + t_{k-1} = t_k.$$
Hence, $v_\zeta$ is continuous and furthermore $v_\zeta'(r)=k^{1/n}$ for $r_{k-1}<r < r_k$ and every $k\geq 1$. In particular, $v_\zeta'$ is unbounded and therefore $x\mapsto u_\zeta(x):=v_\zeta(|x|)$ defines a super-coercive, convex function on $\Rn$. This gives
\begin{align*}
\int_{\Rn}\zeta\big(u_\zeta(x)\big) \d x &= n v_n \int_0^{+\infty} r^{n-1} \zeta\big(v_\zeta(r)\big) \d r\\
&= n v_n \sum_{k=1}^{+\infty} \int_{r_{k-1}}^{r_k} r^{n-1} \zeta\big(v_\zeta(r)\big) \d r = \begin{vmatrix}
t=v_\zeta(r)\\
\d t/\d r = k^{1/n}
\end{vmatrix}\\
&= n v_n \sum_{k=1}^{+\infty} \int_{t_{k-1}}^{t_k} \left(t\frac{1}{k^{1/n}}+r_{k-1} - \frac{t_{k-1}}{k^{1/n}}\right)^{n-1} \zeta(t) \frac{1}{k^{1/n}} \d t,
\end{align*}
where $v_n$ is the volume of the $n$-dimensional unit ball. By the definition of $t_k$ this expression diverges if $r_{k-1}-\frac{t_{k-1}}{k^{1/n}}\geq 0$. Hence, it remains prove this inequality, which we will do by induction on $k$. The statement is obviously true for $k=1$, since $r_0=t_0=0$. Furthermore, it is easy to see, that $r_1=t_1$ and hence the statement also holds true for $k=2$. Assume now, that the statement holds for a $k\in\N$, that is $r_{k-1}-\frac{t_{k-1}}{k^{1/n}}\geq 0$. By definition of $r_k$ we have
\begin{align*}
r_k-\frac{t_k}{(k+1)^{1/n}} &= \frac{t_k-t_{k-1}}{k^{1/n}}+r_{k-1} - \frac{t_k}{(k+1)^{1/n}}\\
&= t_k \left(\frac{1}{k^{1/n}}-\frac{1}{(k+1)^{1/n}}\right) + r_{k-1}-\frac{t_{k-1}}{k^{1/n}},
\end{align*}
which is positive by the induction hypothesis.
\end{proof}

\begin{lemma}[\!\!\cite{mussnig_advances}, Lemma 4.3]
\label{le:int_is_a_val_polar}
For a continuous function $\zeta:\R\to[0,\infty)$ with finite moment of order $n-1$, the map
$$u\mapsto \int_{\dom u^*} \zeta\big(u^*(x)\big) \d x$$
defines a non-negative, continuous, $\SLn$ and dually translation invariant valuation on\linebreak$\CVoin$ that is homogeneous of degree $-n$.
\end{lemma}

\begin{lemma}[\!\!\cite{mussnig_advances}, Lemma 4.6]
\label{le:polar_vol_is_a_val}
For a continuous function $\zeta:\R\to\R$ such that $\zeta(t) =0$ for all $t\geq T$ with some $T\in\R$, the map
$$u\mapsto \int_{\dom u^*} \zeta\big(\nabla u^*(x)\cdot x - u^*(x)\big) \d x$$
defines a continuous, $\SLn$ and translation invariant valuation on $\CVcf$ that is homogeneous of degree $-n$.
\end{lemma}

\begin{lemma}[\!\!\cite{mussnig_advances}, Lemma 4.12]
\label{le:polar_vol_is_a_val_polar}
For a continuous function $\zeta:\R\to\R$ such that $\zeta(t)=0$ for all $t\geq T$ with some $T\in\R$, the map
$$u\mapsto \int_{\dom u} \zeta\big(\nabla u(x) \cdot x - u(x)\big)\d x$$
defines a continuous, $\SLn$ and dually translation invariant valuation on \linebreak$\CVs \cap \CVoin$ that is homogeneous of degree $n$.
\end{lemma}

\section{Super-Coercive Approximations}

The main idea of the proof of Theorem~\ref{thm:main_result} is to utilize a sequence of real-valued functions that can be used to embed $\CVc$ into $\CVs$. We will define and study this sequence in the following.

For $k\in\N$ let $\sumf_k=\sum_{i=1}^k i!$ be the sum of the first $k$ factorials and let $g_k:\R\to\R$ be defined as
$$g_k(r)=\begin{cases}
r,\quad & r\leq \sumf_k\\
\sumf_k + (k+1) (r-\sumf_k),\quad & \sumf_k < r \leq \sumf_k+k!\\
\sumf_{k+j} + (k+j+1)(r-\sumf_{k+j-1}-k!),\quad & \sumf_{k+j-1}+k!< r\leq \sumf_{k+j}+k!,\;j\in \N
\end{cases}
$$
or equivalently
$$g_k(r)=\begin{cases}
r,\quad & r\leq \sumf_k\\
\sumf_{k+j} + (k+j+1)(r-\sumf_{k+j-1}- k!),\quad & \sumf_{k+j-1} + k!< r\leq \sumf_{k+j} + k!,\;j\in\N_0.
\end{cases}
$$
We will need the following properties of the sequence $g_k$.

\begin{lemma}
\label{le:g_k_properties}
For the sequence $g_k:\R\to\R$, the following properties hold true for every $k\in\N$:
\begin{enumerate}[(i)]
	\item $g_k(\sumf_{k+j-1}+ k!)=\sumf_{k+j}$ for every $j\in \N_0$.
	\item $g_k$ is continuous.
	\item\label{it:strictly_incr} $g_k$ is strictly increasing and strictly convex.
	\item $g_k(r)\to r$ as $k\to +\infty$ for every $r\in\R$.
	\item If $u\in\CVc$,  then $g_k(u)\in\CVs$ and furthermore, if $u\in\CVcf$, then $g_k(u)\in\CVsf$.
	\item $g_k(u(x))\geq u(x)$ for every $u\in\CVc$ and $x\in\Rn$.
	\item\label{it:lvl_set} For $u\in\CVc$ and $s\leq \sumf_{k}$ we have $\{g_k(u)\leq s\}=\{u\leq s\}$ and for $\sumf_{k+j-1} +k!< s \leq \sumf_{k+j} + k!$, $j\in\N_0$ we have $\{g_k(u)\leq s\} = \left\{u \leq \frac{s-\sumf_{k+j}}{k+j+1} + \sumf_{k+j-1}+ k!\right\}$.
	\item $g_k(u)\eto u$ as $k\to+\infty$ for every $u\in\CVc$.
	
	\item For every translation $\tau$ on $\Rn$, $\phi\in\SLn$ and $u\in\CVc$, $g_k\circ(u\circ \tau^{-1})=(g_k\circ u)\circ \tau^{-1}$ and $g_k\circ(u\circ \phi^{-1})=(g_k\circ u)\circ \phi^{-1}$.
	
	\item $g_k(u\mx v) = g_k(u)\mx g_k(v)$ and $g_k(u\mn v)=g_k(u)\mn g_k(v)$ for every $u,v\in\CVc$.
	
	\item If $u_j\eto u$ in $\CVc$, then $g_k(u_j)\eto g_k(u)$ as $j\to +\infty$.
\end{enumerate}
\end{lemma}
\begin{proof}
\begin{enumerate}[(i)]
	\item This follows directly from the definition of $g_k$ since
	\begin{align*}
		g_k(\sumf_{k+j-1} + k!) &= \sumf_{k+j-1} + (k+j)(\sumf_{k+j-1}+ k! - \sumf_{k+j-2}-k!)\\
		&= \sumf_{k+j-1} + (k+j)((k+j-1)!)\\
		&= \sumf_{k+j-1} + (k+j)!\\
		&= \sumf_{k+j}
	\end{align*}
	
	\item Since $\sum_{i=0}^{\infty} i! = \infty$, it follows that for every $k\in\N$ and $r\in\R$ either $r\leq \sumf_k + k!$ or there exists $j\in\N$ such that $\sumf_{k+j-1} + k! < r \leq \sumf_{k+j} + k!$.
Furthermore, it is easy to check that for $j\in\N$
\begin{equation*}
\lim_{r\to \sumf_{k+j-1} + k!} g_k(r) = \sumf_{k+j}.
\end{equation*}
Hence, $g_k$ is continuous since $g_k(\sumf_{k+j-1} + k!) = \sumf_{k+j}$.

	\item This is easy to see, since $g_k$ is a continuous, piecewise linear function with positive and increasing slope.
	
	\item This property is immediate, since for every $r\in\R$ there exists $k_0\in\N$ such that $r \leq \sumf_k$ for every $k\geq k_0$ and therefore $g_k(r)=r$.
	
	\item Since $g_k$ is increasing and convex,
\begin{align*}
(g_k\circ u)(\lambda x + (1-\lambda)y) &\leq g_k(\lambda u(x)+(1-\lambda) u(y))\\
&\leq \lambda g_k(u(x)) + (1-\lambda)g_k(u(y)),
\end{align*}
for every $x,y\in\Rn$ and $0\leq \lambda \leq 1$, which shows that $g_k(u)$ is a convex function. Furthermore, since $\lim_{r\to+\infty} \frac{g_k(r)}{r}=+\infty$ for every $k\in\N$, the function $g_k(u)$ is super-coercive. The claim now follows, since $\dom g_k(u)=\dom u$.

	\item This can be easily seen, since $g_k(r)=r$ for every $r\leq \sumf_k$ and $g_k$ is a strictly increasing, convex function.
	
	\item This follows directly from the definition of $g_k$.
	
	\item This follows from the last property together with Lemma \ref{le:hd_conv_lvl_sets}.
	
	\item This is immediate.
	
	\item This is a direct consequence of the monotonicity of $g_k$.
	
	\item This follows from (\ref{it:lvl_set}) together with Lemma \ref{le:hd_conv_lvl_sets}.
\end{enumerate}
\vspace{-1.2em}
\end{proof}

By property (\ref{it:strictly_incr}) of the last Lemma the function $g_k$ is strictly increasing. Hence, there exists an inverse function $g_k^{-1}:\R\to\R$. Furthermore, the particular choice of $g_k$ allows us to construct a function $v_l^t$ as in the next lemma. 

\begin{lemma}
\label{le:g_k_inv_on_v_l_t}
For every $t\in\R$ there exists a sequence of functions $v_l^t:[0,\infty)\to\R$, $l\in\N$ such that the functions $x\mapsto g_k^{-1}(v_l^t(|x|))$ and $x\mapsto v_l^t(|x|)$ are convex, super-coercive and finite on $\Rn$ for every $k\in\N$. Furthermore
$$\elim_{l\to\infty} g_k^{-1}(v_l^t(|\cdot|)) = \Ind_{B^n}+g_k^{-1}(t).$$
\end{lemma}

The construction of $v_l^t$ together with the proof of Lemma~\ref{le:g_k_inv_on_v_l_t} can be found in the Appendix.

\section{Classification of Valuations on $\CVsf$}

The basic idea of the proof of our main result is to embed $\CVc$ into $\CVsf$ by using the sequence $g_k$ that was introduced in the last section and applying Theorem~\ref{thm:class_advances}.

\begin{lemma}
\label{le:special_construction}
Let $(t_k)_{k\in\N}$ and $(b_k)_{k\in\N}$ be strictly monotone sequences of real numbers such that $b_k> 0$ for all $k\in\N$ and $\lim_{k\to\infty} t_k = \lim_{k\to\infty} b_k = \infty$. Furthermore, let $v:[0,\infty)\to\R$ be the piecewise linear function such that $v(0)=t_1$ and
$$v'(r)=b_k$$
if $r>0$ is such that $t_k < v(r) < t_{k+1}$. If $u\in\CVs$ is defined by $u(x)=v(|x|)$ then
$$\nabla u^*(x)\cdot x - u^*(x) = \begin{cases} t_1\quad & \text{for a.e.\! } x \text{ s.t. } |x| < b_1\\
t_k\quad & \text{for a.e.\! } x \text{ s.t. } b_{k-1} < |x| < b_k,\;k\geq 2.\end{cases}$$
\end{lemma}
\begin{proof}
Since $u^*$ is a convex function, it is differentiable a.e.\! on the interior of its domain and since $u$ is super-coercive, it follows from Lemma~\ref{le:conjugate_coercive} that $\dom u^*=\Rn$. Hence, w.l.o.g.\! let $x\in\Rn$ be such that $\nabla u^*(x)$ exists. By Lemma~\ref{le:conjugate_subgradients}
$$\nabla u^*(x)\cdot x - u^*(x) = u(\nabla u^*(x))$$
and furthermore $x \in \partial u (\nabla u^*(x))$. If $x$ is such that for some $k\geq 2$, $b_{k-1}<|x| < b_k$, then this can be only the case if $u(\nabla u^*(x)) = t_k$ and if $|x|<b_1$, then this can only be the case if $\nabla u^*(x)=0$ and $u(\nabla u^*(x))=t_1$.
\end{proof}

\begin{lemma}
\label{le:growth_function_sequences}
For $n\geq 2$, let $\oZ:\CVsf\to[0,\infty)$ be a continuous, $\SLn$ and translation invariant valuation. For $k\in\N$ there exist continuous functions $\zeta_0^k,\zeta_1^k,\zeta_2^k:\R\to[0,\infty)$ such that $\zeta_1^k$ has finite moment of order $n-1$ and $\zeta_2^k(t)=0$ for every $t\geq T_k$ with some $T_k\in\R$ such that
$$\oZ\big(g_k(u)\big) = \zeta_0^k\big(\min\nolimits_{x\in\Rn} u(x)\big) + \int_{\Rn} \zeta_1^k\big(u(x)\big) \d x + \int_{\dom u^*} \zeta_2^k\big(\nabla u^*(x) \cdot x - u^*(x)\big) \d x$$
for every $u\in\CVcf$. Furthermore, the limits
$$\lim_{k\to\infty} \int_{\Rn} \zeta_1^k\big(u(x)\big) \d x\quad \text{and} \quad \lim_{k\to\infty} \int_{\Rn} \zeta_2^k\big(\nabla u^*(x) \cdot x - u^*(x)\big) \d x
$$
exist and are finite for every $u\in\CVsf$. Moreover,
\begin{eqnarray*}
\zeta_0^k\big(\min\nolimits_{x\in\Rn} u(x)\big) &=& \zeta_0\big(\min\nolimits_{x\in\Rn} g_k(u(x))\big)\\
\int_{\Rn} \zeta_1^k\big(u(x)\big) \d x &=& \lim_{m\to\infty} \int_{\Rn} \zeta_1^m\big(g_k(u(x))\big) \d x\\
\int_{\dom u^*} \zeta_2^k\big(\nabla u^*(x) \cdot x - u^*(x)\big) \d x &=& \lim_{m\to\infty} \int_{\Rn} \zeta_2^m\big(\nabla (g_k \circ u)^*(x) \cdot x - (g_k \circ u)^*(x)\big) \d x
\end{eqnarray*}
for every $k\in\N$ and $u\in\CVcf$, where $\zeta_0:\R\to[0,\infty)$ is a continuous function such that $\zeta_0^k(t) = \zeta_0(t)$ for every $t\leq \sum_{i=1}^k i!$.
\end{lemma}
\begin{proof}
By Lemma~\ref{le:g_k_properties} the map
$$u\mapsto \oZ\big(g_k(u)\big)$$
defines a continuous, $\SLn$ and translation invariant valuation on $\CVcf$ for every $k\in\N$. Hence, by Theorem~\ref{thm:class_advances} there exist continuous functions $\zeta_0^k, \zeta_1^k, \zeta_2^k:\R\to[0,\infty)$ such that $\zeta_1^k$ has finite moment of order $n-1$ and $\zeta_2^k(t)=0$ for every $t\geq T_k$ with some $T_k\in\R$ such that
$$\oZ\big(g_k(u)\big) = \zeta_0^k\big(\min\nolimits_{x\in\Rn} u(x)\big) + \int_{\Rn} \zeta_1^k\big(u(x)\big) \d x + \int_{\dom u^*} \zeta_2^k\big(\nabla u^*(x) \cdot x - u^*(x)\big) \d x$$
for every $u\in\CVcf$.

Next, fix an arbitrary $v\in\CVsf$ and let $t_0=\min_{x\in\Rn} v(x)$. Furthermore, for $\lambda >0$ let $v_\lambda(x)=v(x/\lambda)$ for $x\in\Rn$. Note, that by Lemma~\ref{le:g_k_properties} we have $g_k(v_\lambda)\eto v_\lambda$ as $k\to\infty$. Hence, by the continuity of $\oZ$, Lemma~\ref{le:int_is_a_val} and Lemma~\ref{le:polar_vol_is_a_val}
\begin{align*}
\oZ(v_\lambda) &= \lim_{k\to\infty} \oZ\big(g_k(v_\lambda)\big)\\
&= \lim_{k\to\infty} \left(\zeta_0^k(t_0)+\int_{\Rn} \zeta_1^k\big(v_\lambda(x)\big)\d x + \int_{\Rn} \zeta_2^k\big(\nabla v_\lambda^*(x) \cdot x - v_\lambda^*(x)\big) \d x \right)\\
&= \lim_{k\to\infty} \left(\zeta_0^k(t_0)+\lambda^n \int_{\Rn} \zeta_1^k\big(v(x)\big)\d x + \lambda^{-n} \int_{\Rn} \zeta_2^k\big(\nabla v^*(x) \cdot x - v^*(x)\big) \d x \right).
\end{align*}
In particular, the limit on the right-hand side exists and is finite. Considering linear combinations of the last equation with different values of $\lambda$ shows that
$$\lim_{k\to\infty} \int_{\Rn} \zeta_1^k\big(v(x)\big) \d x \quad \text{and} \quad \lim_{k\to\infty} \int_{\Rn} \zeta_2^k\big(\nabla v^*(x) \cdot x - v^*(x)\big) \d x$$
exist and are finite. Moreover, the limit $\lim_{k\to\infty} \zeta_0^k(t_0)$ exists and is finite for every $t\in\R$. Since $v\in\CVsf$ and therefore also $t_0\in\R$ were arbitrary, there exists a function $\zeta_0:\R\to[0,\infty)$ such that $\zeta_0^k\to\zeta_0$ pointwise as $k\to\infty$. Since $t\mapsto \oZ(u+t)$ is continuous for every $u\in\CVsf$ the function $\zeta_0$ must be continuous as well.

Next, let $u\in\CVcf$ and $k\in\N$ be arbitrary. By Lemma~\ref{le:g_k_properties} we have $g_k(u)\in\CVsf$ and therefore
\begin{align*}
\zeta_0^k\big(\min\nolimits_{x\in\Rn} u(x)\big) &+ \int_{\Rn} \zeta_1^k\big(u(x)\big) \d x + \int_{\dom u^*} \zeta_2^k\big(\nabla u^*(x) \cdot x - u^*(x)\big) \d x\\
&= \oZ\big(g_k(u)\big)\\
&= \lim_{m\to\infty} \oZ\big(g_m(g_k(u))\big)\\
&= \lim_{m\to\infty} \zeta_0^m\big(\min\nolimits_{x\in\Rn} g_k(u(x))\big) + \lim_{m\to\infty} \int_{\Rn} \zeta_1^m\big(g_k(u(x))\big) \d x\\
&\quad + \lim_{m\to\infty} \int_{\Rn} \zeta_2^m\big(\nabla(g_k\circ u)^*(x) \cdot x - (g_k\circ u)^*(x)\big) \d x.
\end{align*} 
By homogeneity and the definition of $\zeta_0$ we therefore obtain
\begin{eqnarray*}
\zeta_0^k\big(\min\nolimits_{x\in\Rn} (u(x))\big) &=& \zeta_0\big(\min\nolimits_{x\in\Rn} g_k(u(x))\big)\\
\int_{\Rn} \zeta_1^k\big(u(x)\big) \d x &=& \lim_{m\to\infty} \int_{\Rn} \zeta_1^m\big(g_k(u(x))\big) \d x\\
\int_{\dom u^*} \zeta_2^k\big(\nabla u^*(x) \cdot x - u^*(x)\big) \d x &=& \lim_{m\to\infty} \int_{\Rn} \zeta_2^m\big(\nabla (g_k \circ u)^*(x) \cdot x - (g_k \circ u)^*(x)\big) \d x.
\end{eqnarray*}
\end{proof}

In the following we will call the sequences $\zeta_1^k, \zeta_2^k$ that appear in Lemma~\ref{le:growth_function_sequences} the \emph{growth function sequences}  of the valuation $\oZ$.

\begin{lemma}
\label{le:zeta_1}
For $n\geq 2$, let $\oZ:\CVsf\to[0,\infty)$ be a continuous, $\SLn$ and translation invariant valuation with growth function sequence $\zeta_1^k$, $k\in\N$. There exists a continuous function $\zeta_1:\R\to[0,\infty)$ such that
$$\zeta_1\big(g_k(t)\big)=\zeta_1^k(t)$$
for every $k\in\N$ and $t\in\R$.
\end{lemma}
\begin{proof}
Let $t\in\R$ be arbitrary. Lemma~\ref{le:g_k_inv_on_v_l_t} shows that $x\mapsto g_k^{-1}(v_l^t(|x|)) \in\CVsf$ for every $k,l\in\N$ and by Lemma~\ref{le:growth_function_sequences} we have
$$
\int_{\Rn} \zeta_1^k\big(g_k^{-1} (v_l^t(|x|))\big) \d x = \lim_{m\to\infty} \int_{\Rn} \zeta_1^m\big(g_k(g_k^{-1}(v_l^t(|x|)))\big) \d x =\lim_{m\to\infty} \int_{\Rn} \zeta_1^m\big(v_l^t(|x|)\big) \d x.
$$
Since $g_k^{-1}(v_l^t(|\cdot|)) \eto \Ind_{B^n}+g_k^{-1}(t)$ we therefore have by Lemma~\ref{le:int_is_a_val}
\begin{align*}
V_n(B^n) \zeta_1^k\big(g_k^{-1}(t)\big) &= \int_{\Rn} \zeta_1^k\big(\Ind_{B^n}(x) + g_k^{-1}(t)\big) \d x\\
&= \lim_{l\to\infty} \int_{\Rn} \zeta_1^k\big(g_k^{-1}(v_l^t(|x|))\big) \d x\\
&= \lim_{l\to\infty} \lim_{m\to\infty} \int_{\Rn} \zeta_1^m\big(v_l^t(|x|)\big) \d x.
\end{align*}
Since the right-hand side of this equation is independent of $k$ and only depends on $t$, this defines a non-negative, continuous function $\zeta_1:\R\to[0,\infty)$ such that $\zeta_1(t)=\zeta_1^k(g_k^{-1}(t))$ or equivalently $\zeta_1(g_k(t)) = \zeta_1^k(t)$ for every $t\in\R$.
\end{proof}

For a continuous, $\SLn$ and translation invariant valuation $\oZ:\CVsf\to[0,\infty)$, we call the function $\zeta_1$ from Lemma~\ref{le:zeta_1} the \emph{volume growth function} of $\oZ$.

\begin{lemma}
\label{le:lim_zeta_1}
For $n\geq 2$, let $\oZ:\CVsf\to[0,\infty)$ be a continuous, $\SLn$ and translation invariant valuation. The volume growth function $\zeta_1$ has finite moment of order $n-1$ and furthermore
$$\lim_{k\to\infty} \int_{\Rn} \zeta_1^k\big(u(x)\big) \d x = \int_{\Rn} \zeta_1\big(u(x)\big) \d x$$
for every $u\in\CVsf$.
\end{lemma}
\begin{proof}
Assume that $\zeta_1$ does not have finite moment of order $n-1$. By Lemma~\ref{le:no_moment_no_finite} there exists $u_\zeta\in\CVsf$ such that $\int_{\Rn} \zeta_1(u_\zeta(x)) \d x = +\infty$. For $k\in\N$ let $A_k:=\{u_\zeta \leq \sum_{i=1}^k k!\}$. Note, that by the properties of $u_\zeta$ we have $\bigcup_{k=1}^{\infty} A_k = \Rn$. Furthermore, by Lemma~\ref{le:g_k_properties} and Lemma~\ref{le:zeta_1} we have $\zeta_1^k(u_\zeta(x)) = \zeta_1(g_k(u(x))) = \zeta_1(u(x))$ for every $x\in A_k$ and therefore
$$
\lim_{k\to\infty} \int_{\Rn} \zeta_1^k\big(u_\zeta(x)\big) \d x = \lim_{k\to\infty} \left( \int_{A_k} \zeta_1\big(u_\zeta(x)\big) \d x + \int_{\Rn \backslash A_k} \zeta_1^k\big(u_\zeta(x)\big) \d x \right)
$$
which must be finite by Lemma~\ref{le:growth_function_sequences}. Since both integrals on the right-hand side are non-negative for every $k\in\N$ and $\int_{A_k} \zeta_1(u_\zeta(x)) \d x$ is increasing in $k$, the limit
$$\lim_{k\to\infty} \int_{A_k} \zeta_1\big(u_\zeta(x)\big) \d x$$
exists and is finite, which contradicts the choice of $u_\zeta\in\CVsf$. Hence, $\zeta_1$ must have finite moment of order $n-1$.

By Lemma~\ref{le:int_is_a_val} the map
$$u\mapsto \int_{\Rn} \zeta_1\big(u(x)\big) \d x$$
defines a continuous valuation on $\CVsf$ and therefore
\begin{align*}
\lim_{k\to\infty} \int_{\Rn} \zeta_1^k\big(u(x)\big) \d x = \lim_{k\to\infty} \int_{\Rn} \zeta_1\big(g_k(u(x)\big) \d x = \int_{\Rn} \zeta_1\big(u(x)\big) \d x
\end{align*}
since $g_k(u) \eto u$ for every $u\in\CVsf$.
\end{proof}

\begin{lemma}
\label{le:sup_t_k}
For $n\geq 2$, let $\oZ:\CVsf\to[0,\infty)$ be a continuous, $\SLn$ and translation invariant valuation with growth function sequence $\zeta_2^k$, $k\in\N$. There exists $T\in\R$ such that
$$\zeta_2^k(t)=0$$
for every $k\in\N$ and $t\geq T$.
\end{lemma}
\begin{proof}
We will prove the statement by contradiction and assume that there exists a subsequence $\zeta_2^{k_j}$ and monotone increasing numbers $t_{k_j}\in\R$, $j\in\N$ with $\lim_{j\to\infty} t_{k_j}=+\infty$ such that ${0 < \zeta_2^{k_j}(t_{k_j}) <1}$, which is possible by the properties of $\zeta_2^{k_j}$.
By possibly restricting to another subsequence we can choose the numbers $t_{k_j}$ as follows. Let $t_{k_1}\in\R$ be arbitrary and set $a_{k_1}=0$. If $t_{k_j}$ and $a_{k_j}$ are given, let
$$a_{k_{j+1}} = \sqrt[n]{\frac{j}{v_n \zeta_2^{k_j}(t_{k_j})} + a_{k_j}^n}$$
where $v_n$ is the volume of the $n$-dimensional unit ball and choose $t_{k_{j+1}}$ large enough such that
$$t_{k_{j+1}}-t_{k_j} \geq \max \{1,a_{k_{j+1}}\}.$$ This implies that $t_{k_j}$ and $a_{k_j}$ are strictly monotone increasing sequences such that $\lim_{j\to\infty} t_{k_j} = \lim_{j\to\infty} a_{k_j}= \infty$ and furthermore
\begin{equation}
\label{eq:w_finite}
\frac{t_{k_{j+1}} - t_{k_j}}{a_{k_{j+1}}} \geq \frac{a_{k_{j+1}}}{a_{k_{j+1}}} = 1.
\end{equation}
Next, let $w:[0,\infty)\to\R$ be the piecewise affine function such that $w(0)=t_{k_1}$ and $w'(r)=a_{k_{j+1}}$ for every $r\in[0,\infty)$ with $t_{k_j} < w(r) < t_{k_{j+1}}$, $j\in\N$. Note, that it follows from \eqref{eq:w_finite} that $\sum_{j=1}^{\infty} \frac{t_{k_{j+1}}-t_{k_j}}{a_{k_{j+1}}}=\infty$, which ensures that $w$ is well defined and finite. Furthermore, since $a_{k_j}$ is strictly increasing with $\lim_{j\to\infty} a_{k_j}=\infty$, the function $w$ is super-coercive. If $u\in\CVsf$ is such that $u(x)=w(|x|)$ for every $x\in\Rn$, then by Lemma~\ref{le:special_construction}
$$\nabla u^*(x) \cdot x - u^*(x) = t_{k_j}$$
for a.e.\! $x\in\Rn$ such that $a_{k_j} < |x| < a_{k_{j+1}}$, $j\in\N$. Since the maps $\zeta_2^{k_j}$ are non-negative this gives
\begin{align}
\begin{split}
\label{eq:int_geq_j}
\int_{\Rn} \zeta_2^{k_j}\big(\nabla u^*(x) \cdot x - u^*(x)\big) \d x &\geq \int_{a_{k_j} < |x| < a_{k_{j+1}}} \zeta_2^{k_j}(t_{k_j}) \d x\\
&= v_n(a_{k_{j+1}}^n - a_{k_j}^n) \zeta_2^{k_j}(t_{k_j})\\
&= v_n\left(\frac{j}{v_n \zeta_2^{k_j}(t_{k_j})} + a_{k_j}^n - a_{k_j}^n \right) \zeta_2^{k_j}(t_{k_j})\\
&= j
\end{split}
\end{align}
for every $j\in\N$.

On the other hand, the limit
$$ \lim_{k\to\infty} \int_{\Rn} \zeta_2^k\big(\nabla u^*(x) \cdot x - u^*(x)\big) \d x$$
exists and is finite by Lemma~\ref{le:growth_function_sequences}, which contradicts \eqref{eq:int_geq_j}. Hence, the initial assumption must be false.
\end{proof}

\begin{lemma}
\label{le:lim_zeta_2}
For $n\geq 2$, let $\oZ:\CVsf\to[0,\infty)$ be a continuous, $\SLn$ and translation invariant valuation with growth function sequence $\zeta_2^k$, $k\in\N$. There exists a continuous function $\zeta_2:\R\to[0,\infty)$ such that $\zeta_2(t)=0$ for every $t\geq T$ with some $T\in\R$ and
$$\lim_{k\to\infty} \int_{\Rn} \zeta_2^k\big(\nabla u^*(x) \cdot x - u^*(x)\big) \d x = \int_{\Rn} \zeta_2\big(\nabla u^*(x) \cdot x - u^*(x)\big) \d x$$
for every $u\in\CVsf$.
\end{lemma}
\begin{proof}
Let $T\in\R$ be as in Lemma~\ref{le:sup_t_k} and let $k_0\in\N$ be such that $T+1 \leq \sumf_{k_0}$. Furthermore, for $t\leq T$ let $u_t(x)=|x|+t$ for $x\in\Rn$. Note, that $u_t^*=\Ind_B-t$. Moreover, by the definition of $u_t$ and $g_k$ we can write $g_k(u_t(x))=w_k^t(|x|)$ with a piecewise linear function $w_k^t:[0,\infty)\to\R$ such that $w_k^t(0)=t$, $(w_k^t)'(r)=1$ for every $r>0$ such that $t < w_k^t(r) < \sumf_k$ and $(w_k^t)'(r) \geq k+1$ for a.e. every $r$ such that $w_k^t(r) > \sumf_k > T$ for every $k \geq k_0$. Hence, by Lemma~\ref{le:special_construction} we have for every $k\geq k_0$
$$\nabla(g_k\circ u_t)^*(x)\cdot x - (g_k\circ u_t)^*(x)= t$$
for a.e.\! $x\in\Rn$ with $|x|<1$ and furthermore
$$\nabla(g_k\circ u_t)^*(x)\cdot x - (g_k\circ u_t)^*(x) > T$$
for a.e.\! $x\in\Rn$ with $|x|>1$. Therefore, by Lemma~\ref{le:growth_function_sequences}
\begin{align*}
V_n(B^n) \zeta_2^k(t) &= \int_{B^n} \zeta_2^k\big(\Ind_{B^n}(x) +t\big) \d x\\
&= \int_{\dom u_t^*} \zeta_2^k\big(\nabla u_t^*(x)\cdot x - u_t^*(x)\big) \d x\\
&= \lim_{m\to\infty} \int_{\Rn} \zeta_2^m\big(\nabla(g_k \circ u_t)^*(x)\cdot x - (g_k\circ u_t)^*(x)\big) \d x\\
&= \lim_{m\to\infty} V_n(B^n) \zeta_2^m(t)
\end{align*}
for every $t\leq T$ and every $k\geq k_0$. Since $\zeta_2^k(t)=0$ for every $t>T$ and $k\in\N$, this shows that the sequence $\zeta_2^k$ does not change for $k\geq k_0$. Hence, there exists a function $\zeta_2:\R\to[0,\infty)$ such that
$$\zeta_2(t)=\lim_{m\to\infty} \zeta_2^m(t) = \zeta_2^k(t)$$
for every $k\geq k_0$ and $t\in\R$. In particular, $\zeta_2$ is continuous and $\zeta_2(t)=0$ for every $t\geq T$. Furthermore,
$$\lim_{k\to\infty} \int_{\Rn} \zeta_2^k\big(\nabla u^*(x) \cdot x - u^*(x)\big)\d x = \int_{\Rn} \zeta_2\big(\nabla u^*(x)\cdot x - u^*(x)\big) \d x$$
for every $u\in\CVsf.$
\end{proof}

\subsection{Proof of Theorem~\ref{thm:main_result}}
By Theorem~\ref{thm:class_advances}, equation~\eqref{eq:transl_inv_case} defines a continuous, $\SLn$ and translation invariant valuation on $\CVsf$. 

Conversely, let $\oZ:\CVsf\to[0,\infty)$ be a continuous, $\SLn$ and translation invariant valuation. By Lemma~\ref{le:growth_function_sequences} there exist continuous functions $\zeta_0^k, \zeta_1^k, \zeta_2^k:\R\to[0,\infty)$ such that $\zeta_1^k$ has finite moment of order $n-1$ and $\zeta_2^k(t)=0$ for every $t\geq T_k$ with some $T_k\in\R$ such that
$$\oZ\big(g_k(u)\big)=\zeta_0^k\big(\min\nolimits_{x\in\Rn} u(x)\big) + \int_{\Rn} \zeta_1^k\big(u(x)\big) \d x + \int_{\dom u^*} \zeta_2^k\big(\nabla u^*(x) \cdot x - u^*(x)\big) \d x$$
for every $u\in\CVcf$ and $k\in\N$. By continuity of $\oZ$ and Lemma~\ref{le:g_k_properties}
\begin{align*}
\oZ(u)&=\lim_{k\to\infty} \oZ\big(g_k(u)\big)\\
&=\lim_{k\to\infty}\left( \zeta_0^k\big(\min\nolimits_{x\in\Rn} u(x)\big) + \int_{\Rn} \zeta_1^k\big(u(x)\big) \d x + \int_{\Rn} \zeta_2^k\big(\nabla u^*(x) \cdot x - u^*(x)\big) \d x\right)
\end{align*}
for every $u\in\CVs$. By Lemma~\ref{le:growth_function_sequences} there exists a continuous function $\zeta_0:\R\to[0,\infty)$ such that
$$\lim_{k\to\infty} \zeta_0^k\big(\min\nolimits_{x\in\Rn} u(x)\big) = \zeta_0\big(\min\nolimits_{x\in\Rn} u(x)\big)$$
for every $u\in\CVsf$. By Lemma~\ref{le:lim_zeta_1} there exists a continuous function $\zeta_1:\R\to[0,\infty)$ that has finite moment of order $n-1$ such that
$$\lim_{k\to\infty} \int_{\Rn} \zeta_1^k\big(u(x)\big) \d x = \int_{\Rn} \zeta_1\big(u(x)\big) \d x$$
for every $u\in\CVsf$. Furthermore, by Lemma~\ref{le:lim_zeta_2}
$$\lim_{k\to\infty} \int_{\Rn} \zeta_2^k\big(\nabla u^*(x) \cdot x - u^*(x)\big) \d x = \int_{\Rn} \zeta_2\big(\nabla u^*(x) \cdot x - u^*(x)\big) \d x$$
for some continuous function $\zeta_2:\R\to[0,\infty)$ such that $\zeta_2(t)=0$ for every $t\geq T$ with some $T\in\R$. Hence, $\oZ$ must be as in \eqref{eq:transl_inv_case}.
\hfill\qedsymbol

\subsection{Proof of Theorem~\ref{cor:main_result}}
Lemma~\ref{le:min_is_a_val_polar}, Lemma~\ref{le:int_is_a_val_polar} and Lemma~\ref{le:polar_vol_is_a_val_polar} show that \eqref{eq:dually_transl_inv_case} defines a continuous, $\SLn$ and dually translation invariant valuation on $\CVsf$.

Conversely, let $\oZ:\CVsf\to[0,\infty)$ be a continuous, $\SLn$ and dually translation invariant valuation. By Lemma~\ref{le:conjugate_transl_sln_inv}, Lemma~\ref{le:conjugate_is_a_val}, Lemma~\ref{le:conjugate_coercive} and Theorem~\ref{thm:wijsman} the map $u\mapsto \oZ^*(u):=\oZ(u^*)$ defines a continuous, $\SLn$ and translation invariant valuation on $\CVsf$. Hence, by Theorem~\ref{thm:main_result} and Lemma~\ref{le:u_starstar_is_u}
\begin{align*}
\oZ(u)&=\oZ\big((u^*)^*\big)\\
&= \oZ^*(u^*)\\
&= \widetilde{\zeta_0}\big(\min\nolimits_{x\in\Rn} u^*(x)\big) + \int_{\Rn} \zeta_1\big(u^*(x)\big) \d x + \int_{\Rn} \zeta_2\big(\nabla (u^*)^*(x) \cdot x - (u^*)^*(x)\big) \d x\\
&= \widetilde{\zeta_0}(-u(0)\big) + \int_{\Rn} \zeta_1\big(u^*(x)\big) \d x + \int_{\Rn} \zeta_2\big(\nabla u(x) \cdot x - u(x)\big)\d x
\end{align*}
for every $u\in\CVsf$, where $\widetilde{\zeta_0},\zeta_1,\zeta_2:\R\to[0,\infty)$ are continuous functions such that $\zeta_1$ has finite moment of order $n-1$ and $\zeta_2(t)=0$ for every $t\geq T$ with some $T\in\R$. The statement now follows by setting $\zeta_0(t)=\widetilde{\zeta_0}(-t)$ for $t\in\R$.
\hfill\qedsymbol

\section*{Appendix}
We will give the construction of the function $v_l^t$ from Lemma~\ref{le:g_k_inv_on_v_l_t} and discuss its properties.

By definition of the function $g_k$, $k\in\N$ we can write its inverse function $g_k^{-1}$ as
\begin{equation}
\label{eq:g_k_inv}
g_k^{-1}(s) = \begin{cases}
s,\quad & s \leq \sumf_k\\
\sumf_{k+j-1} + k!+\frac{s-\sumf_{k+j}}{k+j+1},\quad & \sumf_{k+j} < s \leq \sumf_{k+j+1},\; j\in\N_0.
\end{cases}
\end{equation}
Next, for $t\in\R$ let $m_t = \min \{m\in\N \, : \, t\leq \sumf_m \}$ and let
$$A_{l,m}^t = 1+ \frac{\sumf_{m_t} - t + (m-m_t)}{l}$$
for $m\geq m_t$. Note, that $A_{l,m}^t -A_{l,m-1}^t = \frac 1l$ and therefore $\lim_{m\to\infty} A_{l,m}^t=+\infty$. For $l\in\N$ we define the piecewise linear function $v_l^t:[0,\infty)\to \R$ as
\begin{equation}
\label{eq:v_l_t}
v_l^t(r)=
\begin{cases}
t,\quad & 0 \leq r \leq 1\\
t+ l(r-1),\quad & 1 < r \leq A_{l,m_t}^t\\
\sumf_m + (m+1)! l\left(r- A_{l,m}^t\right), \quad & A_{l,m}^t < r \leq A_{l,m+1}^t,\; m \geq m_t.
\end{cases}
\end{equation}

Note, that by this definition
\begin{align}
\begin{split}
\label{eq:v_l_t_at_a_l_m_t}
v_l^t\left(A_{l,m}^t\right) &= \sumf_{m-1} +  m!l  \left(A_{l,m}^t - A_{l,m-1}^t\right)\\
&= \sumf_m\\
&= \lim_{r\to (A_{l,m}^t)^-} v_l^t(r)
\end{split}
\end{align}
for every $m\geq m_t$, $l\in\N$ and $t\in\R$. Hence, $v_l^t$ is continuous. Moreover, it follows immediately that $v_l^t$ is convex, increasing, super-coercive and finite on $[0,\infty)$ and in particular $x\mapsto v_l^t(|x|)\in\CVsf$. Furthermore, by Lemma~\ref{le:hd_conv_lvl_sets} it is easy to see that $\elim_{l\to\infty} v_l^t(|\cdot|) = \Ind_{B^n}+t$.

Next, we will consider the composition $g_k^{-1}\circ v_l^t$. Therefore, fix $t\in\R$ and $k\in\N$. If $t \leq \sumf_k$ we have by \eqref{eq:v_l_t} and \eqref{eq:v_l_t_at_a_l_m_t} that $v_l^t(r) \leq \sumf_k$ for every $r \leq A_{l,k}^t$. Hence, by \eqref{eq:g_k_inv}
$$
g_k^{-1}(v_l^t(r)) = \begin{cases} t,\quad & 0\leq r \leq 1\\
t+l(r-1),\quad & 1 < r \leq A_{l,m_t}^t\\
\sumf_m+(m+1)!l\left(r-A_{l,m}^t\right), \quad & A_{l,m}^t < r\leq A_{l,m+1}^t,\;m_t \leq m < k.
\end{cases}
$$
Furthermore, if $A_{l,m}^t < r\leq A_{l,m+1}^t,\; k \leq m$ we have by \eqref{eq:v_l_t_at_a_l_m_t} that $\sumf_m < v_l^t(r) \leq \sumf_{m+1}$. Therefore, using \eqref{eq:g_k_inv} with $k+j=m$ gives
\begin{align*}
g_k^{-1}(v_l^t(r)) & = g_k^{-1}\left(\sumf_m + (m+1)! l \big(r-A_{l,m}^t\big)\right)\\
&= \sumf_{m-1} + k! + \frac{\sumf_m + (m+1)! l \big(r-A_{l,m}^t\big) - \sumf_m}{m+1}\\
&= \sumf_{m-1} + k!+ m! l \left(r-A_{l,m}^t\right)
\end{align*}
for $A_{l,m}^t < r\leq A_{l,m+1}^t$ and  $k \leq m$. In particular
\begin{align*}
g_k^{-1} (v_l^t(A_{l,k}^t)) &= \sumf_{k-1} + k! l \left(A_{l,k}^t-A_{l,k-1}^t\right)\\
&= \sumf_k\\
&= \lim_{r\to (A_{l,k}^t)^-} g_k^{-1}(v_l^t(r))
\end{align*}
which shows that $g_k^{-1}\circ v_l^t$ is continuous. Furthermore,
$$\frac{\d}{\d r}\, g_k^{-1}(v_l^t(r))= k! l$$
for $A_{l,k-1}^t < r < A_{l,k+1}^t$. In particular, the slope of $g_k^{-1}\circ v_l^t$ is increasing, despite the fact that the slope of $g_k^{-1}$ is decreasing. Hence, it is easy to see, that $g_k^{-1}\circ v_l^t$ is a convex, increasing, super-coercive and finite function on $[0,\infty)$. Moreover, if follows from Lemma~\ref{le:hd_conv_lvl_sets} that $\elim_{l\to\infty} g_k^{-1}(v_l^t(|\cdot|)) = \Ind_{B^n}+t = \Ind_{B^n}+g_k^{-1}(t)$.

In the case $t>\sumf_k$, there exists $j_t\in\N_0$ such that $\sumf_{k+j_t} < t \leq \sumf_{k+j_t+1}$ and therefore, similarly to the case above,
$$
g_k^{-1}(v_l^t(r)) = \begin{cases}
\sumf_{k+j_t-1} + k!+\frac{t-\sumf_{k+j_t}}{k+j_t+1},\quad & 0 \leq r \leq 1\\
\sumf_{k+j_t-1} + k!+\frac{t-\sumf_{k+j_t}+l(r-1)}{k+j_t+1},\quad & 1 < r \leq A_{l,m_t}^t\\
\sumf_{m-1}+ k!+m! l \left(r-A_{l,m}^t\right), \quad & A_{l,m}^t < r \leq A_{l,m+1}^t,\;m_t \leq m.
\end{cases}
$$
Again, this is a convex, increasing, super-coercive and finite function on $[0,\infty)$ with\linebreak$\elim_{l\to\infty} g_k^{-1}(v_l^t(|\cdot|)) = \Ind_{B^n}+g_k^{-1}(t)$.

\phantomsection
\addcontentsline{toc}{section}{Acknowledgments}
\section*{Acknowledgments}

The author would like to thank Andrea Colesanti for pointing out the idea of Lemma~\ref{le:no_moment_no_finite}.

\footnotesize

\phantomsection

\bigskip\bigskip
\parindent 0pt\footnotesize

\parbox[t]{8.5cm}{
Fabian Mussnig\\
Institut f\"ur Diskrete Mathematik und Geometrie\\
Technische Universit\"at Wien\\
Wiedner Hauptstra\ss e 8-10/1046\\
1040 Wien, Austria\\
e-mail: fabian.mussnig@tuwien.ac.at}

\end{document}